\documentclass[11pt]{article}
\usepackage{amsmath,amsthm,amssymb,enumerate,mathscinet}
\usepackage{fullpage}
\usepackage{graphicx}

\newtheorem{theorem}{Theorem}[section]

\newtheorem{lemma}[theorem]{Lemma}

\newtheorem{thm}[theorem]{Theorem}
\newtheorem{prop}[theorem]{Proposition}

\numberwithin{equation}{section}

\def\F{\mathcal{F}}

\def\P{\mathcal{P}}
\def\S{\mathcal{S}}

\def\eps{\varepsilon}

\def\Sh{\text{Sh}}
\def\VC{\text{VC}}
\def\ExVC{\text{ExVC}}
\def\Conn{\text{Conn}}
\def\IndMat{\text{IndMat}}

\def\COMMENT#1{}
\let\COMMENT=\footnote

\title{Two results about the hypercube} 

\author{J\'ozsef Balogh\footnote{Department of Mathematical Sciences,
 University of Illinois at Urbana-Champaign, Urbana, Illinois 61801, USA, {\tt
jobal@math.uiuc.edu}.Research is partially supported by NSF Grant DMS-1500121 and Arnold O. Beckman Research Award (UIUC Campus Research Board 15006) and the Langan Scholar Fund (UIUC).
},
~Tam\'{a}s M\'{e}sz\'{a}ros\footnote{Freie Universit\"at Berlin, 14195 Berlin, Germany, {\tt tamas.meszaros@fu-berlin.de}. Research is supported by the DRS POINT Fellowship Program at Freie Universit\"at Berlin.},
 ~and Adam Zsolt Wagner\footnote{University of Illinois at Urbana-Champaign, Urbana, Illinois 61801, USA, {\tt
zawagne2@illinois.edu}. }}

\begin{document}
\maketitle
\begin{abstract}
First we consider families in the hypercube $Q_n$ with bounded VC dimension. Frankl raised the problem of estimating the number $m(n,k)$ of maximal families of VC dimension $k$. Alon, Moran and Yehudayoff showed that $$n^{(1+o(1))\frac{1}{k+1}\binom{n}{k}}\leq m(n,k)\leq n^{(1+o(1))\binom{n}{k}}.$$ We close the gap by showing that $\log \left(m(n,k)\right)= {(1+o(1))\binom{n}{k}}\log n$ and show how a tight asymptotic for the logarithm of the number of induced matchings between two adjacent small layers of $Q_n$ follows as a corollary.

Next, we consider the integrity $I(Q_n)$ of the hypercube, defined as $$I(Q_n) = \min\{ |S| +m(Q_n \setminus S) : S \subseteq V (Q_n) \},$$ where $m(H)$ denotes the number of vertices in the largest connected component of $H$. Beineke, Goddard, Hamburger, Kleitman, Lipman and Pippert showed that $c\frac{2^n}{\sqrt{n}} \leq I(Q_n)\leq C\frac{2^n}{\sqrt{n}}\log n$ and suspected that their upper bound is the right value. We prove that the truth lies below the upper bound by showing that $I(Q_n)\leq  C \frac{2^n}{\sqrt{n}}\sqrt{\log n}$.

\end{abstract}

\section{Introduction}

Throughout the paper we will use standard notation. For an integer $n\geq 1$ we will write $[n]$ for the set $\{1,2,\dots,n\}$, $\P(n)$ for its power set and ${\binom{[n]}{k}}$ for the collection of subsets of size $k$. For $\bigcup\limits_{i=0}^k{\binom{[n]}{i}}$ (resp.~$\sum\limits_{i=0}^k{\binom{n}{i}}$) we will use the shorthand notation ${\binom{[n]}{\leq k}}$ (resp.~${\binom{n}{\leq k}}$). 

\vspace{0.3cm}

For an integer $n\geq 1$ the graph $Q_n$, the hypercube of dimension $n$, has vertex set $V(Q_n)=\{0,1\}^n$ and two vertices are connected if they only differ in one coordinate. There is a natural bijection between the vertex set of $Q_n$ and $\P(n)$, and we will use them interchangeably. Here we consider several enumerative and extremal properties of vertex subsets of this graph.

\subsection{Enumerative problems}
We say that a family $\F\subseteq\P(n)$ \emph{shatters} a set $S\subseteq [n]$ if for all $A\subseteq S$ there exists a set $B\in \F$ with $B\cap S = A$. Let $\Sh(\F):=\{S\subseteq [n]: \F \text{ shatters }S\}$. The \emph{Vapnik-Chervonenkis dimension}, \emph{VC-dimension} for short, of a family $\F\subseteq \P(n)$ is defined as
\begin{equation*}
\VC(\F)=\max\{|S|: \F \text{ shatters }S\}.
\end{equation*}
Pajor's version~\cite{pajor} of the  Sauer-Shelah lemma states that  we always have $|\Sh(\F)|\geq |\F|$. We say that a family $\F\subseteq \P(n)$ is \emph{(shattering-)extremal}  if $|\Sh(\F)|=|\F|$. For example, if $\F$ is a down-set then it is extremal, simply because in this case $\Sh(\F)=\F$. For an integer $k\geq 0$ let $\ExVC(n,k)$ be the number of extremal families in $\P(n)$ with VC dimension at most $k$. The study of these extremal families was initiated by Bollob\'as, Leader and Radcliffe \cite{blr} and since then many interesting results have been obtained in connection with these combinatorial objects. 

The Sauer-Shelah lemma~\cite{sauershelah} states that for any family $\F\subseteq \P(n)$ we have $|\F|\leq \binom{n}{\leq \VC(\F)}$. A family is called \emph{maximal} if $|\F|=\binom{n}{\leq \VC(\F)}$. Clearly, every maximal family is extremal. Frankl~\cite{frankl} raised the question of estimating $m(n,k)$, the number of maximal families in $\P(n)$ of VC dimension $k$, and showed that
\begin{equation*}
2^{\binom{n-1}{k}}\leq m(n,k)\leq 2^{n\binom{n-1}{k}}.
\end{equation*}
Alon, Moran and Yehudayoff~\cite{alonmoran} showed that for constant $k\geq 2$ we have, as $n\rightarrow \infty$, that
\begin{equation}\label{alonbound}
n^{(1+o(1))\frac{1}{k+1}\binom{n}{k}}\leq m(n,k)\leq n^{(1+o(1))\binom{n}{k}}.
\end{equation}
We close the gap and show that the upper bound of~(\ref{alonbound}) is correct, even if we allow $k$ to grow as $k=n^{o(1)}$.

\begin{thm}\label{franklcounting}
Let $k=n^{o(1)}$. Then $m(n,k)= n^{(1+o(1))\binom{n}{k}}$.
\end{thm}

A \emph{matching} in a graph $G = (V, E)$ is a set of edges $M \subseteq E$ without any common vertices. An \emph{induced matching} is a matching such that no endpoints of two edges of $M$ are joined by an edge of $G$. For an integer $k\geq 0$ let $\binom{[n]}{k}$ be the collection of those vertices of $Q_n$ which contain precisely $k$ ones. We will refer to these collections for different values of $k$ as the layers of the hypercube $Q_n$. Let further $\IndMat(n,k)$ be the number of induced matchings in $Q_n$ between the layers $\binom{[n]}{k}$ and $\binom{[n]}{k+1}$. Our next result concerns the quantities $\IndMat(n,k)$ and $\ExVC(n,k)$.

\begin{thm}\label{franklcor}
Let $k=n^{o(1)}$. Then $$n^{(1+o(1))\binom{n}{k}}\leq \IndMat(n,k)\leq m(n,k)\leq \ExVC(n,k)\leq n^{(1+o(1))\binom{n}{k}}.$$
\end{thm}

\subsection{The integrity of $Q_n$}

Next we consider the problem of finding a small family $\F\subseteq \{0,1\}^n$ such that all connected components of $Q_n\setminus \F$ are small. For a graph $H$ let $m(H)$ denote the maximum number of vertices in a component of $H$. The \emph{integrity} $I(G)$ of a graph $G$, introduced by Barefoot, Entringer and Swart~\cite{barefoot} to measure the vulnerability of a network, is defined by 
\begin{equation*}
I(G) = \min\{ |S| +
m(G \setminus S) : S \subseteq V (G) \}.
\end{equation*}
In~\cite{falseconj} it was conjectured that for the hypercube we have $I(Q_n)=2^{n-1}+1$, but Beineke, Goddard, Hamburger, Kleitman, Lipman and Pippert~\cite{kleitman} disproved this conjecture and obtained the following bounds:
\begin{thm}[\label{kleitmanthm}Beineke, Goddard, Hamburger, Kleitman, Lipman, Pippert]
There exists constants $c,C>0$ such that $$c\frac{2^n}{\sqrt{n}}\leq I(Q_n)\leq C\frac{2^n}{\sqrt{n}}\log n.$$
\end{thm}
Their upper bound was obtained by a series of `orthogonal' cuts and they suspected it to be of the correct order of magnitude. We show that the true value of $I(Q_n)$ lies below their upper bound.
\begin{thm}\label{integrity}
There exists a constant $C>0$ such that the integrity of the hypercube satisfies $$I(Q_n)\leq C\frac{2^n}{\sqrt{n}}\sqrt{\log n}.$$
\end{thm}

\vspace{0.3cm}

This note is organized as follows. We prove Theorems~\ref{franklcounting} and~\ref{franklcor} in Section~\ref{enumproof} and Theorem~\ref{integrity} in Section~\ref{integrityproof}. We often omit floor and ceiling signs when they are not crucial, to increase the clarity of our presentation.

\section{The proofs of Theorems~\ref{franklcounting} and~\ref{franklcor}}\label{enumproof}

\begin{prop}\label{countmaximals}
If $k=n^{o(1)}$ then $m(n,k)\geq \IndMat(n,k)\geq n^{(1+o(1))\binom{n}{k}}.$
\end{prop}
\begin{proof}
Let $\S$ be the collection of all induced matchings between layers $\binom{[n]}{k}$ and $\binom{[n]}{k+1}$ and let $M(n,k)$ be the collection of all maximal families $\F\subset \P(n,k)$ of VC-dimension $k$. We will first define an injection $\phi: \S \rightarrow M(n,k)$ and then show $|\S|\geq n^{(1+o(1))\binom{n}{k}}$.

Given an element $S\in\S$, define $\phi(S)$ as follows: $\phi(S)$ contains all sets of size at most $k-1$, those sets of size $k$ which are not covered by edges in $S$ and those sets of size $k+1$ which are covered by edges in $S$. Observe that we can reconstruct $S$ from $\phi(S)\cap\left(\binom{[n]}{k}\cup \binom{[n]}{k+1}\right)$ hence we have $\phi(S_1)\neq \phi(S_2)$ for any $S_1,S_2\in\S$ with $S_1\neq S_2$. As $|\phi(S)|=\binom{n}{\leq k}$ it remains to show that $\VC(\phi(S))=k$ for all $S\in\S$.

As $\phi(S)\subset \binom{[n]}{\leq k+1}$ we have $VC(\phi(S))\leq k+1$. Suppose for contradiction that $\phi(S)$ shatters a set $A\in\binom{[n]}{k+1}$. Then since $\phi(S)\subset \binom{[n]}{\leq k+1}$ we must have $A\in\phi(S)$. This means that  the induced matching $S$ meets $A$, let the other endpoint of this edge be $B\in\binom{[n]}{k}$. Hence $B\notin \phi(S)$ and as $S$ is induced,  for all $A'\in\binom{[n]}{k+1}$ with $A'\neq A$ and $B\subset A'$ we have $A'\notin \phi(S)$. So there is no set $C\in\phi(S)$ with $C\cap A=B$, contradicting the assumption that $A$ is shattered by $\phi(S)$.

Now we turn to showing that $|\S|\geq n^{(1+o(1))\binom{n}{k}}$. Let $\eps=\eps(n)=o\left(1/k^2\right)$ be a sequence converging to zero as $n\rightarrow\infty$.  Let $A=\{1,2,\ldots, \eps n\}$ and $B=[n]\setminus A$. Call a matching $M$ between layers $\binom{[n]}{k}$ and $\binom{[n]}{k+1}$ \emph{good} if for every edge $(C_1,C_2)$ of $M$ we have that $C_1\in \binom{B}{k}$ and $C_2\in \binom{[n]}{k+1}\setminus \binom{B}{k+1}$. Every good matching is induced, and the number of good matchings that cover $\binom{B}{k}$ is already 
\begin{equation*}
(\eps n)^{\binom{(1-\eps)n}{k}}\geq \left(n^{1-o(1)}\right)^{(1-2\eps)^k\binom{n}{k}}=n^{(1-o(1))\binom{n}{k}}.
\end{equation*}
This completes the proof.
\end{proof}

\begin{proof}[Proof of Theorem~\ref{franklcounting}]
Let $k=n^{o(1)}$. By Proposition~\ref{countmaximals} we have $n^{(1+o(1))\binom{n}{k}}\leq m(n,k)$ and by~(\ref{alonbound}) we have $n^{(1+o(1))\binom{n}{k}}\geq m(n,k)$.
\end{proof}

\begin{proof}[Proof of Theorem~\ref{franklcor}]
Let $k=n^{o(1)}$.  For integers $n,m\geq 0$ let $\Conn(n,m)$ be the number of connected induced subgraphs of $Q_n$ on exactly $m$ vertices. First, notice that $m(n,k)\leq \ExVC(n,k)$ as every maximal family is also extremal. Second, as every extremal family induces a connected subgraph of $Q_n$ (see e.g.~\cite{greco}) and as by the Sauer inequality any family of VC dimension $k$ has size at most $\binom{n}{\leq k}$ we have 
\begin{equation*}
\ExVC(n,k)\leq \sum_{i=0}^{\binom{n}{\leq k}}\Conn(n,i).
\end{equation*}
We now proceed following the exact same ideas as the ones described in~\cite{alonmoran}. It is known (e.g., Problem~$45$ in \cite{bollobas}) that the number of connected subgraphs of size $\ell$ in a graph of order $N$ and maximum degree $D$ is at most $N(e(D-1))^{\ell-1}\leq N(eD)^{\ell}$. In our case, plugging in $\ell = i$, $N = 2^n$ and $D = n$ yields
\begin{equation*}
\ExVC(n,k)\leq \sum_{i=0}^{\binom{n}{\leq k}}2^n(en)^{i}\leq 2^{n+1} (en)^{\binom{n}{\leq k}}=n^{(1+o(1)){\binom{n}{k}}},
\end{equation*}
where for the last equality we used that $k=n^{o(1)}$. Hence together with Proposition~\ref{countmaximals} we have that
\begin{equation*}
n^{(1+o(1))\binom{n}{k}}\leq \IndMat(n,k)\leq m(n,k)\leq \ExVC(n,k)\leq n^{(1+o(1))\binom{n}{k}}.
\end{equation*}
\end{proof}

\section{The proof of Theorem~\ref{integrity}}\label{integrityproof}

Our goal in this section is to improve the upper bound from Theorem~\ref{kleitmanthm}. In~\cite{kleitman} a set was formed by a series of orthogonal cuts which yielded the upper bound in Theorem~\ref{kleitmanthm}. Instead, we will delete small spheres around some appropriately chosen points.

For a vertex $x$ of $Q_n$ and a non-negative integer $r$ for the ball of radius $r$ with center $x$ write
\begin{equation*}
B_{r}(x)=\{y\in V(Q_n)\ |\ d(y,x)\leq r\}
\end{equation*}
and for the sphere of radius $r$ with center $x$
\begin{equation*}
S_{r}(x)=\{y\in V(Q_n)\ |\ d(y,x)= r\},
\end{equation*}
where $d(.,.)$ denotes the Hamming distance. 

Let us define $\alpha$ and $r_0$ by the equations 
\begin{equation*}
\frac{1}{e^{2\alpha^2}\alpha}=\frac{\sqrt{\log n}}{\sqrt{n}}\quad\text{and}\quad r_0 := \Big\lfloor\frac{n}{2}-\alpha\sqrt{n}\Big\rfloor.
\end{equation*}
This $r_0$ will be the radius of the spheres we delete. Note that $\alpha$ is very close to $\sqrt{\log n}/2$. The proof of Theorem~\ref{integrity} will be based on the following standard estimates, we provide a full proof for completeness at the end of this section. 
\begin{lemma}\label{easylemma}
For every $x\in \P(n)$
\begin{equation*}
|B_{r_0}(x)|=\binom{n}{\leq r_0}=\Theta\left(\frac{2^n\sqrt{\log n}}{\sqrt{n}}\right)\quad \text{and}\quad |S_{r_0}(x)|\frac{\sqrt{n}}{\sqrt{\log n}}=\binom{n}{r_0}\frac{\sqrt{n}}{\sqrt{\log n}}=\Theta \left(\frac{2^n\sqrt{\log n}}{\sqrt{n}}\right).
\end{equation*}
\end{lemma}

The proof of Theorem~\ref{integrity} follows by repeatedly removing spheres of radius $r_0$ around some appropriately chosen points.

\begin{proof}[Proof of Theorem~\ref{integrity}]
For a family $\F\subset \P(n)$ and $x\in\P(n)$ let $B(\F,x):=B_{r_o}(x)\cap \F$ and $S(\F,x):=S_{r_0}(x)\cap \F$. Fixing the family $\F$ and picking $x\in\P(n)$ uniformly at random,  using linearity of expectation and Lemma~\ref{easylemma}, for the expectations of $|B(\F,x)|$ and $|S(\F,x)|$ we get
\begin{equation*}
E_x(|B(\F,x)|)=\Theta\left(|\F|\frac{\sqrt{\log n}}{\sqrt{n}}\right)\quad \text{and}\quad E_x(|S(\F,x)|)=\Theta \left(|\F|\frac{\log n}{n}\right).
\end{equation*}
In particular there is an absolute constant $C>0$ such that for every $n$ and $\emptyset\neq\F\subseteq\P(n)$ there exists an $x=x(\F)$ satisfying 
\begin{equation*}
|S(\F,x)|\leq |B(\F,x)|\frac{C\sqrt{\log n}}{\sqrt{n}} .
\end{equation*}
Now to formalize our strategy, fix such a function $x: \P(\P(n))\setminus\{\emptyset\}\rightarrow \P(n)$. Let $\F_0=2^{[n]}$, $S_0=\emptyset$ and for $i=0,1,\ldots$ set $\F_{i+1}:=\F_{i}\setminus B(\F_{i},x(\F_i))$ and $S_{i+1}:=S_{i} \cup S(\F_{i},x(\F_i))$. So,  each time peel off new components by removing some appropriately chosen sphere of radius $r_0$. Let $\ell$ be the least integer such that $\F_\ell=\emptyset$. Then 
\begin{equation*}
|S_\ell|=\sum_{i=0}^{\ell-1}S(\F_i,x(\F_i))\leq \frac{C\sqrt{\log n}}{\sqrt{n}}\sum_{i=0}^{\ell-1}B(\F_i,x(\F_i))=\frac{C\sqrt{\log n}}{\sqrt{n}}2^n.
\end{equation*}
Deleting $S_\ell$ from $Q_n$ leaves a graph with all components being contained in some ball of radius $r_0$, and by Lemma~\ref{easylemma} having size at most $\binom{n}{\leq r_0}=\Theta\left(\frac{2^n\sqrt{\log n}}{\sqrt{n}}\right)$.
\end{proof}

All that remains is to prove Lemma~\ref{easylemma}.

\begin{proof}[Proof of Lemma~\ref{easylemma}]
As for $n$ fixed, the function $f(x)=\binom{n}{x}\binom{n}{x-1}^{-1}=\frac{n-x+1}{x}$ is decreasing we have that
\begin{align*}
&\binom{n}{r_0}\binom{n}{r_0-\frac{\sqrt{n}}{\alpha}}^{-1}=\prod_{i=r_0-\frac{\sqrt{n}}{\alpha}+1}^{r_0}\binom{n}{i}\binom{n}{i-1}^{-1}\geq \left(\binom{n}{r_0}\binom{n}{r_0-1}^{-1}\right)^{\sqrt{n}/\alpha}\\
=&\left(\frac{n-r_0+1}{r_0}\right)^{\sqrt{n}/\alpha}\geq\left(\frac{n/2}{n/2-\alpha\sqrt{n}}\right)^{\sqrt{n}/\alpha}\geq \left(1+\frac{\alpha}{\sqrt{n}}\right)^{\sqrt{n}/\alpha}\geq 1.1.
\end{align*}
Also note that, again using that $f(x)$ is decreasing, the same lower bound holds if we replace $r_0$ by any other value $r$ with $\frac{\sqrt{n}}{\alpha}\leq r\leq r_0$.  

Similarly
\begin{align*}
&\binom{n}{r_0}\binom{n}{r_0-\frac{\sqrt{n}}{\alpha}}^{-1}=\prod_{i=r_0-\frac{\sqrt{n}}{\alpha}+1}^{r_0}\binom{n}{i}\binom{n}{i-1}^{-1}\leq \left(\binom{n}{r_0-\frac{\sqrt{n}}{\alpha}+1}\binom{n}{r_0-\frac{\sqrt{n}}{\alpha}}^{-1}\right)^{\sqrt{n}/\alpha}\\
=& \left(\frac{n-r_0+\frac{\sqrt{n}}{\alpha}}{r_0-\frac{\sqrt{n}}{\alpha}+1}\right)^{\sqrt{n}/\alpha} \leq \left(\frac{n/2+2\alpha\sqrt{n}}{n/2-2\alpha\sqrt{n}}\right)^{\sqrt{n}/\alpha}\leq \left(1+\frac{10\alpha}{\sqrt{n}}\right)^{\sqrt{n}/\alpha}\leq 10^{10}.
\end{align*}
Accordingly,
\begin{equation*}
\binom{n}{\leq r_0}\geq \sum_{i=r_0-\frac{\sqrt{n}}{\alpha}}^{r_0}\binom{n}{i}\geq \frac{\sqrt{n}}{\alpha}\binom{n}{r_0-\frac{\sqrt{n}}{\alpha}}=\frac{\sqrt{n}}{\alpha}\binom{n}{r_0}\left(\binom{n}{r_0-\frac{\sqrt{n}}{\alpha}}\binom{n}{r_0}^{-1}\right)\geq\Omega\left(\binom{n}{r_0}\frac{\sqrt{n}}{\sqrt{\log n}}\right)
\end{equation*}
and
\begin{align*}
&\binom{n}{\leq r_0}=\sum_{j=0}^{r_0 \frac{\alpha}{\sqrt{n}}}\sum_{i=r_0-(j+1)\frac{\sqrt{n}}{\alpha}+1}^{r_0-j\frac{\sqrt{n}}{\alpha}}\binom{n}{i}\leq \sum_{j=0}^{r_0 \frac{\alpha}{\sqrt{n}}} \frac{\sqrt{n}}{\alpha}\binom{n}{r_0-j\frac{\sqrt{n}}{\alpha}}=\frac{\sqrt{n}}{\alpha}\binom{n}{r_0} \sum_{j=0}^{r_0 \frac{\alpha}{\sqrt{n}}}\binom{n}{r_0-j\frac{\sqrt{n}}{\alpha}}\binom{n}{r_0}^{-1}\\
=&\frac{\sqrt{n}}{\alpha}\binom{n}{r_0} \sum_{j=0}^{r_0 \frac{\alpha}{\sqrt{n}}}\prod_{i=0}^{j-1}\binom{n}{r_0-(i+1)\frac{\sqrt{n}}{\alpha}}\binom{n}{r_0-i\frac{\sqrt{n}}{\alpha}}^{-1}\leq \frac{\sqrt{n}}{\alpha}\binom{n}{r_0} \sum_{j=0}^{r_0 \frac{\alpha}{\sqrt{n}}} \left(\frac{1}{1.1}\right)^j=O\left(\binom{n}{r_0}\frac{\sqrt{n}}{\sqrt{\log n}}\right).
\end{align*}
These inequalities together give
\begin{equation*}
\binom{n}{\leq r_0}=\Theta\left(\binom{n}{r_0}\frac{\sqrt{n}}{\sqrt{\log n}}\right).
\end{equation*}
On the other hand we have
\begin{equation*}
\binom{n}{r_0}=\binom{n}{n/2}\prod_{i=r_0}^{n/2-1}\binom{n}{i}\binom{n}{i+1}^{-1}=\binom{n}{n/2}\prod_{i=1}^{\alpha\sqrt{n}}\frac{n/2-i+1}{n/2+i}=\binom{n}{n/2}\prod_{i=1}^{\alpha\sqrt{n}}\left(1-\frac{2i-1}{n/2+i}\right),
\end{equation*}
and hence, using the inequality $1-x\leq e^{-x}$, we get
\begin{equation*}
\binom{n}{r_0}\leq \binom{n}{n/2}\prod_{i=1}^{\alpha\sqrt{n}}e^{-\frac{2i-1}{n/2+i}}=\binom{n}{n/2}e^{-\sum_{i=1}^{\alpha\sqrt{n}}\frac{2i-1}{n/2+i}}
=O\left(e^{-2\alpha^2}\binom{n}{n/2}\right),
\end{equation*}
and similarly, using the inequality $1-x\geq e^{-\frac{x}{1-x}}$ for $0<x<1$, we get 
\begin{equation*}
\binom{n}{r_0}\geq \binom{n}{n/2}\prod_{i=1}^{\alpha\sqrt{n}}e^{-\frac{2i-1}{n/2-i+1}}=\binom{n}{n/2}e^{-\sum_{i=1}^{\alpha\sqrt{n}}\frac{2i-1}{n/2-i+1}}
=\Omega\left(e^{-2\alpha^2}\binom{n}{n/2}\right).
\end{equation*}
Together this gives
\begin{equation*}
\binom{n}{r_0}=\Theta\left(e^{-2\alpha^2}\binom{n}{n/2}\right)=\Theta\left(\frac{\log n}{\sqrt{n}}\cdot\frac{2^n}{\sqrt{n}}\right)=\Theta\left(\frac{2^n\log n}{n}\right),
\end{equation*}
and the claim follows.
\end{proof}

\textbf{Acknowledgement:} We are grateful to Shagnik Das for his many insightful comments regarding the manuscript.

\end{document}